\documentclass[a4paper,12pt]{amsart}

\usepackage{fullpage}

\makeindex

\usepackage{xcolor}
\usepackage[
colorlinks=true,
citecolor=blue,
urlcolor=blue,
linkcolor=blue,
pdfborder={0 0 0},
breaklinks
]{hyperref}

\usepackage{amsfonts,graphics,amsmath,amsthm,amscd,amssymb,amsmath,latexsym,euscript, enumerate,kotex,mathtools}
\usepackage{epsfig,url}
\usepackage{flafter}
\usepackage[all,cmtip,line]{xy}
\usepackage{array}
\usepackage[english]{babel}
\usepackage{overpic}
\usepackage{subfig}
\usepackage{multirow}
\usepackage{tikz-cd}
\usepackage[T1]{fontenc}
\usepackage{lmodern}
\usepackage{pifont}

\usepackage{microtype}
\usepackage{wrapfig}
\usepackage[shortlabels]{enumitem}
\setlist[enumerate, 1]{1\textsuperscript{o}}

\newtheorem{theorem}{Theorem}[section]
\newtheorem{lemma}[theorem]{Lemma}

\theoremstyle{definition} % italic or bold etc.
\newtheorem{definition}[theorem]{Definition}
\newtheorem{definition-lemma}[theorem]{Definition-Lemma}

\theoremstyle{remark}
\newtheorem{remark}[theorem]{Remark}

\numberwithin{equation}{section}

\newcommand{\R}{\mathbb{R}}

\newcommand{\Q}{\mathbb{Q}}

\newcommand{\A}{\mathbb{A}}

\def\Spec{\operatorname{Spec}}

\def\Supp{\operatorname{Supp}}

\newcommand{\ceil}[1]{\left\lceil #1 \right \rceil}
\let\oldframe\frame
\renewcommand\frame[1][allowframebreaks]{\oldframe[#1]}

\title[On the normalized local volume of a non-closed point]{On the normalized local volume of a non-closed point}

\date{\today}

\begin{document}

\author[D.~Kim]{Donghyeon Kim}
\address[Donghyeon Kim]{Department of Mathematics, Yonsei University, 50 Yonsei-ro, Seodaemun-gu, Seoul 03722, Republic of Korea}
\email{narimial0@gmail.com}

\begin{abstract}
In this note, we show that the normalized local volume of a non-closed point can be expressed in terms of the normalized local volumes of closed points.
\end{abstract}

\maketitle
\allowdisplaybreaks

\section{Introduction}
The notion of normalized local volume is introduced in \cite{Li18} to investigate a canonical metric structure of an affine cone. It has been revealed that this notion has connections with Kähler-Einstein metric in \cite{Li17}. Many conjectures are settled around the notion under the name of \emph{stable degeneration conjecture} (cf. \cite{Li18,LLX20,LX20}), and most of the conjectures are now solved (cf. \cite{Blu18,BL21,Xu20,XZ21,XZ25}).

\smallskip

Usually, the notion of normalized local volume is defined in a closed point of a variety. One may wonder if we can extend the notion to a non-closed point (cf. \cite[Conjecture 6.7]{LLX20} and \cite[Remark 2.5]{LX20}). For example, one can define the normalized local volume at a non-closed point as a natural generalization of that in a closed point, as in Definition \ref{doeum}, and one can also define it using the constructibility of normalized local volume (cf. \cite[Theorem 1.3]{Xu20}) for closed points (for example, \cite{HLQ23,Zhu24}). In particular, it was conjectured that after the right scaling, the definitions are equivalent. In this note, we confirm the intuition regarding the normalized local volume of a non-closed point and provide the full affirmative answer on \cite[Conjecture 6.7]{LLX20}.

\begin{theorem}[{cf. \cite[Conjecture 6.7]{LLX20}}] \label{maintheorem}
Let $(X,\Delta)$ be a klt pair of a variety over an uncountable field, and let $x\in X$ be a (not necessarily closed) point. Let $x'\in \overline{\{x\}}$ be a general closed point, and let $n:=\dim X$ and $d:=\dim \overline{\{x\}}$. Then,
$$ \widehat{\mathrm{vol}}(x',X,\Delta)=\frac{n^n}{(n-d)^{n-d}}\cdot \widehat{\mathrm{vol}}(x,X,\Delta).$$
\end{theorem}

Let us provide a sketch of the proof. Intuitively, we want to mimic the argument in \cite[Proof of Theorem 2.16]{Liu22}. Using the henselian pair argument (cf. \cite[Step 3 of Proof of Theorem 1.1]{Kim26q}), we can obtain an étale cover $Y\to X$ that admits a left inverse $Y\to Z$ of $Z\to Y$ (which plays the role of $\pi:X\to \A^1$ in \cite[Proof of Theorem 2.16]{Liu22}); this is sufficient to prove our main theorem using the arguments in \cite{Liu22,LZ21}. Note that for using \cite[Theorem 1.3]{Xu20} in our proof, we must treat the geometric generic fiber of $Y\to Z$, and that the quasi-monomiality and uniqueness of the minimizer of normalized local volume is indispensable (cf. \cite{Xu20,XZ21}).

\smallskip

The note consists of three sections. The second section collects the definitions and lemmas we need, and the third section proves Theorem \ref{maintheorem}.

\section*{Acknowledgement}
The author is partially supported by the Samsung Science and Technology Foundation
under Project Number SSTF-BA2302-03.

\section{Preliminaries}
Let us collect the notions we will use. Let $k$ be a (not necessarily closed) field of characteristic $0$. A \emph{variety} is a separated, irreducible, and finite type scheme over $k$. For more details, see \cite{Fuj17} and \cite{KM98}.

\smallskip

\begin{itemize}
    \item For a point $x\in X$ of a scheme $X$, $\mathfrak{m}_x$ is the maximal ideal of $\mathcal{O}_{X,x}$, and $k(x):=\mathcal{O}_{X,x}/\mathfrak{m}_x$ is the quotient field.
    \item For a morphism $f:Y\to Z$ of varieties and a point $z\in Z$, $Y_z:=Y\times_Z \Spec k(z)$ is the fiber, and $Y_{\overline{z}}:=Y\times_Z \Spec \overline{k(z)}$ is the geometric fiber.
    \item Let $X$ be a normal scheme, and let $\Delta$ be an effective $\Q$-Weil $\Q$-divisor on $X$. Then, $(X,\Delta)$ is a \emph{couple}. Moreover, if $K_X+\Delta$ is $\Q$-Cartier, then $(X,\Delta)$ is a \emph{pair}.
    \item For a pair $(X,\Delta)$, a proper birational morphism $f:X'\to X$ and a prime divisor $E$ on $X'$, we define
    $$ A_{X,\Delta}(E)=\mathrm{mult}_{E}(K_{X'}-f^*(K_X+\Delta))+1.$$
    Note that the definition only depends on $E$, not $f$. Moreover, $(X,\Delta)$ is \emph{Kawamata log terminal (klt)} if $A_{X,\Delta}(E)>0$ for every prime divisor $E$ over $X$.
    \item For a noetherian local ring $R$ with the maximal ideal $\mathfrak{m}$ and an $\mathfrak{m}$-primary ideal $\mathfrak{a}\subseteq R$, let us denote by $\ell_R(R/\mathfrak{a})$ the length of $R/\mathfrak{a}$ as an $R$-module. If $R$ is clear, then we denote $\ell_R(R/\mathfrak{a})$ by $\ell(R/\mathfrak{a})$.
\end{itemize}

\subsection{Valuation}
In this section, let $X$ be a variety over $k$. Let us say that $\nu:K(X)\to \R$ is a \emph{valuation} if
\begin{itemize}
\item[(1)] $\nu(a)=0$ for every $a\in k$,
\item[(2)] $\nu(f+g)\ge \min\{\nu(f),\nu(g)\}$, and
\item[(3)] $\nu(fg)=\nu(f)+\nu(g)$ for every $f,g\in K(X)$.
\end{itemize}
By convention, we set $\nu(0)=\infty$. Recall that the \emph{center} of $\nu$ on $X$, denoted by $c_X(\nu)$, is a point $\xi\in X$ such that $\nu\ge 0$ on $\mathcal{O}_{X,\xi}$, and $\nu>0$ on the maximal ideal. Since $X$ is separated, the center is unique. We say that a valuation $\nu$ is \emph{over} $X$ if the center of $\nu$ is in $X$. We set $\mathrm{Val}_X$ to be the set of valuations over $X$. We endow the weak topology on $\mathrm{Val}_X$.

\smallskip

For every point $x\in X$, $\mathrm{Val}_{X,x}\subseteq \mathrm{Val}_X$ is the subset of valuations centered on $x$. We endow $\mathrm{Val}_{X,x}$ with the subspace topology. A valuation $\nu\in \mathrm{Val}_X$ is \emph{divisorial} if it has rational rank $1$, i.e., $\nu(K(X))\subseteq \R$ is a rank $1$ abelian group. We denote by $\mathrm{DivVal}_X\subseteq \mathrm{Val}_X$ and $\mathrm{DivVal}_{X,x}\subseteq \mathrm{Val}_{X,x}$ the set of divisorial valuations with the subspace topology. We denote by $\mathrm{Val}^*_X,\mathrm{Val}^*_{X,x},\mathrm{DivVal}^*_X,\mathrm{DivVal}^*_{X,x}$ the set of non-trivial valuations.

\smallskip

For a valuation $\nu\in \mathrm{Val}_X$, we define its \emph{valuation ideal sheaf} $\mathfrak{a}_p(\nu)$ for a non-negative real number $p$ as
$$ \mathfrak{a}_p(\nu):=\{f\in \mathcal{O}_X\mid \nu(f)\ge p\}.$$
We also define the \emph{valuation ideal sequence} of $\nu$ as $\mathfrak{a}_{\bullet}(\nu):=\{\mathfrak{a}_m(\nu)\}_{m\ge 0}$.

\smallskip

Let $(X,\Delta)$ be a klt pair, and let $\nu\in \mathrm{Val}_X$ be a valuation. We may assign the \emph{log discrepancy} $A_{X,\Delta}(\nu)$ to $\nu$. For more details, see \cite[Definition 1.34]{Xu25}.

\smallskip

Let $\mathfrak{a}_{\bullet}:=\{\mathfrak{a}_m\}_{m\ge 1}$ be a sequence of ideals in $\mathcal{O}_X$. We say that $\mathfrak{a}_{\bullet}$ is a \emph{graded sequence of ideals} in $\mathcal{O}_X$ if $\mathfrak{a}_m\mathfrak{a}_n\subseteq \mathfrak{a}_{m+n}$ for every positive integers $m,n$. For a klt pair $(X,\Delta)$, we denote by
$$ \mathrm{lct}(X,\Delta,\mathfrak{a}_{\bullet})=\sup_{m\ge 1}m\cdot \mathrm{lct}(X,\Delta,\mathfrak{a}_m),$$
where $\mathrm{lct}(X,\Delta,\mathfrak{a}_m)$ is the log canonical threshold (cf. \cite[Chapter 1.3]{Xu25}).

\subsection{Normalized local volume}
In this subsection, $X$ is a variety over $k$. For a point $x\in X$, define $n:=\dim X$ and $d:=\dim \overline{\{x\}}$. For an $\mathfrak{m}_x$-primary ideal $\mathfrak{a}$ in $\mathcal{O}_X$, we define the \emph{Samuel multiplicity} $\mathrm{mult}(\mathfrak{a})$ to be 
$$\mathrm{mult}(\mathfrak{a}):=\lim_{\ell\to \infty}\frac{\ell_{\mathcal{O}_{X,x}}(\mathcal{O}_{X,x}/\mathfrak{a}^{\ell})}{\frac{\ell^{n-d}}{(n-d)!}}.$$ For a graded sequence of ideals $\mathfrak{a}_{\bullet}$, the \emph{Samuel multiplicity} of $\mathfrak{a}_{\bullet}$ is
$$ \mathrm{mult}(\mathfrak{a}_{\bullet}):=\lim_{k\to \infty}\frac{\ell(\mathcal{O}_{X,x}/\mathfrak{a}_k)}{\frac{k^{n-d}}{(n-d)!}}=\lim_{m\to \infty}\frac{\mathrm{mult}(\mathfrak{a}_m)}{m^{n-d}},$$
where the second equality follows from \cite[Theorem 6.5]{Cut14}. For a valuation $\nu\in \mathrm{Val}_{X,x}$, we define
$$ \mathrm{vol}(\nu):=\lim_{k\to \infty}\frac{\ell(\mathcal{O}_{X,x}/\mathfrak{a}_k(\nu))}{\frac{k^{n-d}}{(n-d)!}}=\mathrm{mult}(\mathfrak{a}_{\bullet}(\nu)).$$
The notion of $\mathrm{vol}(\nu)$ is first introduced in \cite{ELS03}. Note that every limit presented above exists.

\smallskip

Let us define the notion of \emph{normalized local volume}. We will extend the notion to a non-closed setting.

\begin{definition} \label{doeum}
Let $(X,\Delta)$ be a klt pair, and let $x\in X$ be a (not necessarily closed) point. We define
$$ \widehat{\mathrm{vol}}(x,X,\Delta):=\inf_{\nu\in \mathrm{Val}^*_{X,x}}A_{X,\Delta}(\nu)^{\dim X-\dim \overline{\{x\}}}\cdot \mathrm{vol}(\nu).$$
\end{definition}

We also denote by 
$$\widehat{\mathrm{vol}}_{X,\Delta}(\nu):=\begin{cases}A_{X,\Delta}(\nu)^{\dim X-\dim \overline{\{x\}}}\cdot \mathrm{vol}(\nu) & A_{X,\Delta}(\nu)<\infty \\ \infty & A_{X,\Delta}(\nu)=\infty.\end{cases}$$ 
We think that the following lemma is well-known to experts, and we will present the proof for a general non-closed $x\in X$.

\begin{lemma} \label{weolseong}
Let $(X,\Delta)$ be a klt pair, and let $x\in X$ be a point. Then,
$$ \widehat{\mathrm{vol}}(x,X,\Delta)=\inf_{\mathfrak{a}}\mathrm{lct}(X,\Delta,\mathfrak{a})^{\dim X-\dim \overline{\{x\}}}\cdot \mathrm{mult}(\mathfrak{a}),$$
where $\mathfrak{a}$ runs over all $\mathfrak{m}_x$-primary ideals. In particular, we have
$$\widehat{\mathrm{vol}}(x,X,\Delta)=\inf_{\nu\in \mathrm{DivVal}^*_{X,x}}A_{X,\Delta}(\nu)^{\dim X-\dim \overline{\{x\}}}\cdot \mathrm{vol}(\nu).$$
\end{lemma}

\begin{proof}
We will mimic \cite[Proof of Proposition 2.8]{LX20}.

\smallskip

For any valuation $\nu\in \mathrm{Val}^*_{X,x}$, consider the graded sequence of valuative ideals $ \mathfrak{a}_k(\nu) $. Then, $\nu(\mathfrak{a}_k(\nu))\ge k$, and we can estimate:
$$ A_{X,\Delta}(\nu)^{\dim X-\dim \overline{\{x\}}}\cdot \frac{\mathrm{mult} (\mathfrak{a}_k(\nu))}{k^{\dim X-\dim \overline{\{x\}}}}\ge \mathrm{lct}(X,\Delta,\mathfrak{a}_k)^{\dim X-\dim \overline{\{x\}}}\cdot \mathrm{mult}(\mathfrak{a}_k(\nu)).$$
Moreover, by definition, $\mathrm{vol}(\nu)=\lim\limits_{k\to\infty} \frac{\mathrm{mult}(\mathfrak{a}_k(\nu))}{k^{\dim X-\dim \overline{\{x\}} }}$. Hence, we get one direction.

\smallskip

For any $\mathfrak{m}_x$-primary ideal $\mathfrak{a}\subseteq \mathcal{O}_{X,x}$, we can choose a divisorial valuation $\nu$ computing $\mathrm{lct}(X,\Delta,\mathfrak{a})$. Then, $\nu$ is centered on $x$. Assume $\nu(\mathfrak{a})=k$. Then, we have $\mathfrak{a}^{\ell}\subseteq \mathfrak{a}_{k\ell}(\nu)$ for any positive integer $\ell$. So, we can  estimate:
$$ \mathrm{lct}(X,\Delta,\mathfrak{a})^{\dim X-\dim \overline{\{x\}}}\cdot \mathrm{mult}(\mathfrak{a})\ge A_{X,\Delta}(\nu)^{\dim X-\dim \overline{\{x\}}}\cdot \frac{\mathrm{mult}(\mathfrak{a}_{k\ell}(\nu))}{(k\ell)^{\dim X-\dim \overline{\{x\}}}}.$$
As $\ell\to \infty$, we then get another direction. The last claim in the lemma can be followed by the proof.
\end{proof}

\subsection{Izumi's inequality}
Let us prove a version of Izumi's inequality for a non-closed $x\in X$.

\begin{lemma} \label{guryongpo}
Let $(X,\Delta)$ be a klt pair, and let $x\in X$ be a (not necessarily closed) point. Then, there exists $c>0$ such that for every $\nu\in \mathrm{Val}_{X,x}$ and every $f\in \mathcal{O}_{X,x}$,
$$ \nu(f)\le C\cdot A_{X,\Delta}(\nu)\cdot \mathrm{ord}_x(f).$$
\end{lemma}

\begin{proof}
The proof closely follows the argument in \cite[Proof of Lemma 2.14 (2)]{HLQ23}.

\smallskip

Let $\mu:X'\to (X,\Delta)$ be a log resolution of $(X,\Delta)$, and write
$$ K_{X'}+\Delta'=\mu^*(K_X+\Delta).$$
Since $(X,\Delta)$ is klt, there exists $\varepsilon>0$ such that $\Delta'\le (1-\varepsilon)\Delta'_{\mathrm{red}}$. Since $\Delta'_{\mathrm{red}}$ is snc, $(X',\Delta'_{\mathrm{red}})$ is lc. Hence,
$$ A_{X,\Delta}(\nu)\ge \varepsilon\cdot A_{X',0}(\nu).$$
Let $\xi$ be the center of $\nu$ on $X'$. Then, by Izumi's inequality (cf. \cite[Proposition 5.10]{JM12}),
$$ \nu(f)\le \varepsilon^{-1}A_{X,\Delta}(\nu)\cdot \mathrm{ord}_{\xi}(\mu^*f).$$
Applying \cite[Theorem 7.44]{Xu25} gives the proof. Note that \cite[Proof of Theorem 7.44]{Xu25} works for non-closed $x\in X$.
\end{proof}

\subsection{Henselization and Kollár component}
Let $A$ be a noetherian ring, and let $I\subseteq A$ be an ideal. We define the following:

\begin{definition}[{cf. \cite[Lemma 0A02]{Stacks}}]
The \emph{henselization} $(A^h,I^h)$ of $(A,I)$ is
$$ A^h:=\lim_{\substack{\longrightarrow \\ B\in \mathcal{C}}}B,$$
where $\mathcal{C}$ is the category of étale ring maps $A\to B$ such that $A/I\to B/IB$ is an isomorphism, and $I^h:=IA^h$.
\end{definition}

We also define the notion of \emph{Kollár component}. We assume $k$ to be algebraically closed.

\begin{definition}[{cf. \cite{Xu14}}]
Let $x\in (X,\Delta)$ be a klt singularity, and let $S$ be a prime divisor over $X$. If there exists a proper birational morphism $f:X'\to X$ such that $S=f^{-1}(x)$ is the unique exceptional divisor, $(X',S+f^{-1}_*\Delta)$ is plt and $-(K_{X'}+S+f^{-1}_*\Delta)$ is $f$-ample, we call $S$ a \emph{Kollár component} over $x\in (X,\Delta)$.
\end{definition}

\subsection{{$\Q$-Gorenstein family of singularities}}
In this subsection, we will define the notion of \emph{$\Q$-Gorenstein family of singularities}. Note that, unlike \cite[Definition 7]{BL21} and \cite[Definition 2.9]{Xu20}, we do not assume the fibers to be connected. We assume $k$ to be algebraically closed.

\begin{definition} \label{nojeokbong}
Let $(Y,\Delta_Y)$ be a pair, and let $Y\to Z$ be a flat morphism of varieties. Let $\imath:Z\to Y$ be a section. We say that $Z\subseteq (Y,\Delta_Y)\to Z$ is a \emph{$\Q$-Gorenstein family of singularities} if
\begin{itemize}
    \item[(1)] for every $z\in Z$, the connected component of $Y_{\overline{z}}$ containing $\imath(\overline{z})$ is normal, and
    \item[(2)] for every $z\in Z$, $\Supp \Delta_Y$ does not contain $Y_z$.
\end{itemize}
If further
\begin{itemize}
    \item[(3)] $Y_{\overline{z}}$ is irreducible,
\end{itemize}
then $Z\subseteq (Y,\Delta_Y)\to Z$ is a \emph{$\Q$-Gorenstein family of irreducible singularities}, and if
\begin{itemize}
    \item[(4)] $Y_{\overline{z}}$ is irreducible, and $(Y_{\overline{z}},(\Delta_Y)_{\overline{z}})$ is klt for every $z\in Z$,
\end{itemize}
we say that $Z\subseteq (Y,\Delta_Y)\to Z$ is a \emph{$\Q$-Gorenstein family of klt singularities}.
\end{definition}

Note that in many cases, we can reduce the case of $\Q$-Gorenstein family of singularities to the $\Q$-Gorenstein family of irreducible singularities in a very satisfying way.

\begin{lemma} \label{idong}
Let $(Y,\Delta_Y)\to Z$ be a $\Q$-Gorenstein family of singularities, and $\imath:Z\to Y$ be a section of $Y\to Z$. Then, there exists an open subscheme $Y'\subseteq Y$ such that $Y'_{\overline{z}}$ is irreducible and contains $\imath(\overline{z})$ for every $z\in Z$.
\end{lemma}

\begin{proof}
Let us define $Y':=\bigcup_{z\in Z}Y'_z$; here, $Y'_z$ is the connected component containing $\imath(z)$. Then, $Y'\subseteq Y$ is an open subscheme by \cite[Lemma 055R]{Stacks}. Moreover, by \cite[Lemma 055M]{Stacks}, $(Y')_{\overline{z}}$ is connected (thus irreducible because $Y_{\overline{z}}$ is normal) and contains $\imath(\overline{z})$.
\end{proof}

\section{Proof of the main theorem}
In this section, we assume $k$ to be algebraically closed. We believe that the following lemma is well-known among experts (cf. \cite[The next paragraph of Conjecture 6.7]{LLX20}).

\begin{lemma} \label{sambong}
Let $Z\subseteq (Y,\Delta_Y)\to Z$ be a $\Q$-Gorenstein family of klt singularities with $\imath:Z\to Y$ over an uncountable base field. For any general closed point $x'\in Z$, we have
\begin{equation} \label{dalseo}
\widehat{\mathrm{vol}}(\imath(\overline{\eta}),Y_{\overline{\eta}},(\Delta_Y)_{\overline{\eta}})=\frac{(n-d)^{n-d}}{n^n}\widehat{\mathrm{vol}}(\imath(x'),Y,\Delta_Y),
\end{equation}
where $n:=\dim X$ and $d:=\dim Z$.
\end{lemma}

\begin{proof}
We will follow \cite[Proof of Proposition 2.17]{Liu22} almost verbatim.

\smallskip

By applying \cite[Theorem 1.3]{Xu20} to $(Y,\Delta_Y)\to Z$, we have that there exists an open subscheme $U\subseteq Z$ such that
\begin{equation} \label{Xu}
\widehat{\mathrm{vol}}(\imath(x'),Y_{x'},(\Delta_Y)_{x'})=\widehat{\mathrm{vol}}(\imath(\overline{\eta}),Y_{\overline{\eta}},(\Delta_Y)_{\overline{\eta}})
\end{equation}
for every closed point $x'\in U$. Now, the ``$\le$'' part can be derived from \cite[Theorem 1.7]{LZ21} and (\ref{Xu}), and we only need to prove the ``$\ge$'' part. Fix $\varepsilon>0$. Choose a Kollár component $S_{\overline{\eta}}$ over $\imath(\overline{\eta})\in (Y_{\overline{\eta}},(\Delta_Y)_{\overline{\eta}})$ such that 
\begin{equation}\label{agibong}
\widehat{\mathrm{vol}}_{Y_{\overline{\eta}},(\Delta_Y)_{\overline{\eta}}}(S_{\overline{\eta}})\le \widehat{\mathrm{vol}}(\imath(\overline{\eta}),Y_{\overline{\eta}},(\Delta_Y)_{\overline{\eta}})+\varepsilon
\end{equation}
(cf. \cite[Theorem 1.3]{LX20}). Let $\mu:Y'\to Y$ be a birational morphism extracting $S$ over $Z$ such that the geometric generic fiber of $S$ is $S_{\overline{\eta}}$. Such $\mu$ exists if we replace $Z$ by $Z'$ such that there exists a quasi-finite morphism $Z'\to Z$ (cf. \cite[Remark 2.8]{KL25}), and we may assume that $Z'\to Z$ is an étale morphism by the generic smoothness (cf. \cite[Corollary 3.10.7]{Har77}). Note that the normalized local volume is an invariant under any étale morphism (cf. \cite[2.14]{Kol13}, Lemma \ref{weolseong}, and \cite[Proposition 2.24]{HLQ23}).

\smallskip

For $m\ge 0$, define $\mathfrak{b}_p:=\mu_*\mathcal{O}_{Y'}(-\ceil{pS})$. Let us mimic the argument in \cite[Proof of Lemma 4.2]{Smi18}. Since $-S$ is ample, we see that
$$ \bigoplus_{m\ge 0}\mathfrak{b}_m=\bigoplus_{m\ge 0} \mu_*\mathcal{O}_{Y'}(-\ceil{mS}) $$
is a finitely generated $\mathcal{O}_Z$-algebra. Thus, the associated graded algebra $\bigoplus_{m\ge 0}\mathfrak{b}_m/\mathfrak{b}_{m+1}$ on $I=\bigoplus_{m\ge 0}\mathfrak{b}_{m+1}$ is also a finitely generated $\mathcal{O}_Z$-algebra. Hence, by the generic flatness (cf. \cite[Proposition 052A]{Stacks}), after shrinking $Z$, we may assume that the algebra is flat over $\mathcal{O}_Z$. Thus, each piece $\mathfrak{b}_m/\mathfrak{b}_{m+1}$ is flat over $\mathcal{O}_Z$. Considering
$$ 0\to \mathfrak{b}_m/\mathfrak{b}_{m+1}\to \mathcal{O}_{Y}/\mathfrak{b}_{m+1} \to \mathcal{O}_{Y}/\mathfrak{b}_m\to 0,$$
we may assume that for every $m$, $\mathcal{O}_Y/\mathfrak{b}_m$ is flat over $\mathcal{O}_Z$. Furthermore, by
$$ 0\to \mathfrak{b}_m\to \mathcal{O}_Y\to \mathcal{O}_Y/\mathfrak{b}_m\to 0, $$
we may assume further that $\mathfrak{b}_m$ is flat over $\mathcal{O}_Z$ for every $m$.

\smallskip

we may assume that $\mathfrak{b}_m$ is a flat family of ideals over $Z$ after shrinking $Z$. Let $S_{x'}$ and $\mathfrak{b}_{p,x'}$ be the restrictions of $S$ and $\mathfrak{b}_p$ on $Y_{x'}$. By flatness of $S$ and $\mathcal{O}_Y/\mathfrak{b}_p$ over $Z$, we have $\mathfrak{b}_{p,x'}=\mathfrak{a}_p(S_{x'})$ for any $x'\in Z$.
\smallskip

Fix a general $x'\in Z$, and let $t_1,\cdots,t_d\in \mathfrak{m}_{Z,x'}$ be regular parameters. Set 
$$\mathfrak{t}:=(t_1,\cdots,t_d)\mathcal{O}_{Y,\imath(x')}=\mathfrak{m}_{Z,x'}\mathcal{O}_{Y,\imath(x')}.$$ 
Fix $s>0$, and define
$$ \mathfrak{I}_{m,s}:=\sum^m_{\ell=0}\mathfrak{t}^{\ell}\mathfrak{b}_{(m-\ell)s}.$$
Considering
$$\mathfrak{I}_{m,s}=\mathfrak{t}^m+\mathfrak{I}_{m,s}\subsetneq \mathfrak{t}^{m-1}+\mathfrak{I}_{m,s}\subsetneq \cdots \subsetneq \mathfrak{t}+\mathfrak{I}_{m,s}\subsetneq \mathcal{O}_{Y,\imath(x')}, $$
we obtain
$$
\begin{aligned}
\ell(\mathcal{O}_{Y,\imath(x')}/\mathfrak{I}_{m,s})&=\sum^{m-1}_{\ell=0}\ell\left(\frac{\mathfrak{t}^{\ell}+\mathfrak{I}_{m,s}}{\mathfrak{t}^{\ell+1}+\mathfrak{I}_{m,s}}\right) 
\\ &=\sum^{m-1}_{\ell=0}\ell\left(\frac{\mathfrak{t}^{\ell}}{\mathfrak{t}^{\ell}\cap(\mathfrak{t}^{\ell+1}+\mathfrak{I}_{m,s})}\right)
\\ &=\sum^{m-1}_{\ell=0}\ell\left(\frac{\mathfrak{t}^{\ell}}{\mathfrak{t}^{\ell}\cap(\mathfrak{t}^{\ell+1}+\mathfrak{b}_{(m-\ell)s})}\right) & (1)
\\ &=\sum^{m-1}_{\ell=0}\ell\left(\frac{\mathfrak{t}^{\ell}(\mathcal{O}_{Y,\imath(x')}/\mathfrak{b}_{(m-\ell)s})}{\mathfrak{t}^{\ell+1}(\mathcal{O}_{Y,\imath(x')}/\mathfrak{b}_{(m-\ell)s})}\right)
\\ &=\sum^{m-1}_{\ell=0}\ell((\mathfrak{m}^{\ell}_{x'}/\mathfrak{m}^{\ell+1}_{x'})\otimes_{k(x')}(\mathcal{O}_{Y_{x'},\imath(x')}/\mathfrak{b}_{((m-\ell)s)x'})) & (2)
\\&=\sum^{m-1}_{\ell=0}\dim_{k(x')}(\mathfrak{m}^{\ell}_{x'}/\mathfrak{m}^{\ell+1}_{x'})\cdot \ell(\mathcal{O}_{Y_{x'},\imath(x')}/\mathfrak{b}_{((m-\ell)s)x'})
\\ &=\sum^{m-1}_{\ell=0}\binom{\ell+d-1}{d-1}\cdot \ell(\mathcal{O}_{Y_{x'},\imath(x')}/\mathfrak{b}_{((m-\ell)s)x'}),& (3)
\end{aligned}$$
where
\begin{itemize}
    \item (1) and (2) come from the flatness of $\mathcal{O}_{Y,\imath(x')}/\mathfrak{b}_{(m-\ell)s}$ over $\mathcal{O}_{Z,x'}$, and
    \item (3) follows from \cite[Lemma 00NO]{Stacks}.
\end{itemize}

Note that, by the definition of local volume,
\begin{equation} \label{pohanghang}
\begin{aligned}\ell(\mathcal{O}_{Y_{x'},\imath(x')}/&\mathfrak{b}_{((m-\ell)s)x'})
\\ & =\frac{1}{(n-d)!}\mathrm{vol}(\mathrm{ord}_{S_{x'}})(m-\ell)^{n-d}s^{n-d}+O((m-\ell)^{n-d-1}s^{n-d-1}).
\end{aligned}
\end{equation}
Using
$$ \binom{\ell+d-1}{d-1}=\frac{\ell^{d-1}}{(d-1)!}+O(\ell^{d-2})$$
and
\begin{equation} \label{elementary}
\sum^{m-1}_{\ell=0}\ell^{d-1}(m-\ell)^{n-d}=\frac{(d-1)!(n-d)!}{n!}m^n+O(m^{n-1})
\end{equation}
((\ref{elementary}) comes from $\int^1_0x^{d-1}(1-x)^{n-d}\,\mathrm{d}x=\frac{(d-1)!(n-d)!}{n!}$), we obtain that from (\ref{pohanghang}),
$$ \ell(\mathcal{O}_{Y,\imath(x')}/\mathfrak{I}_{m,s})=\frac{s^{n-d}}{n!}\mathrm{vol}(\mathrm{ord}_{S_{x'}})m^n+O(m^{n-1}).$$
Thus, the Samuel multiplicity is 
\begin{equation} \label{haksan}
\mathrm{mult}(\mathfrak{I}_{\bullet,s})=\mathrm{vol}(\mathrm{ord}_{S_{x'}})s^{n-d}
\end{equation}

\smallskip

Let $\nu_s$ be the valuation of $k(X)$ as the quasi-monomial combination of $Y_1,\cdots,Y_d$ and $S$ of weight $s,s,\cdots,1$, where $Y_i$ is the pullback of $(t_i=0)\subseteq Z$ near $x'$. Then, it is clear that
\begin{equation} \label{quasi-monomial combination}
A_{Y,\Delta_Y}(\nu_s)=ds+A_{Y_{x'},(\Delta_Y)_{x'}}(S_{x'}),\text{ and } \nu_s(\mathfrak{I}_{m,s})\ge ms.
\end{equation}
Hence, we have
$$ \mathrm{lct}(Y,\Delta_Y,\mathfrak{I}_{\bullet,s})\le \frac{A_{Y,\Delta_Y}(\nu_s)}{\nu_s(\mathfrak{I}_{\bullet,s})}\le d+s^{-1}A_{Y_{x'},(\Delta_Y)_{x'}}(S_{x'}).$$
Combining \cite[Theorem 27]{Liu18} and (\ref{quasi-monomial combination}) with (\ref{haksan}), we obtain
$$ 
\begin{aligned}
\widehat{\mathrm{vol}}(\imath(x'),Y,\Delta_Y)&\le \mathrm{lct}(Y,\Delta_Y,\mathfrak{I}_{\bullet,s})^n\cdot \mathrm{mult}(\mathfrak{I}_{\bullet,s})
\\&\le (d+s^{-1}A_{Y_{x'},(\Delta_U)_{x'}}(S_{x'}))^n\cdot \mathrm{vol}(\mathrm{ord}_{S_{x'}})s^{n-d}.
\end{aligned}$$
Since $s$ is arbitrary, we may choose $s=\frac{A_{Y_{x'},(\Delta_Y)_{x'}}(S_{x'})}{n-d}$, and hence
$$ 
\begin{aligned}\widehat{\mathrm{vol}}(\imath(x'),Y,\Delta_Y)&\le \frac{n^n}{(n-d)^{n-d}}A_{Y_{x'},(\Delta_Y)_{x'}}(S_{x'})^{n-d}\cdot \mathrm{vol}(\mathrm{ord}_{S_{x'}})
\\ &\le \frac{n^n}{(n-d)^{n-d}}(\widehat{\mathrm{vol}}(\imath(\overline{\eta}),Y_{\overline{\eta}},(\Delta_Y)_{\overline{\eta}})+\varepsilon).
\end{aligned}
$$
The second inequality follows from (\ref{agibong}) and
$$ \widehat{\mathrm{vol}}_{Y_{x'},\imath(x')}(\mathrm{ord}_{S_{x'}})=\widehat{\mathrm{vol}}_{Y_{\overline{\eta}},\imath(\overline{\eta})}(\mathrm{ord}_{S_{\overline{\eta}}})$$
by flatness of $S$ and $\mathfrak{b}_{\bullet}$ over $Z$. Hence, letting $\varepsilon\to 0^+$ and (\ref{Xu}) gives
\begin{equation} \label{namduiduiaegu}
\widehat{\mathrm{vol}}(\imath(x'),Y,\Delta_Y)\le \frac{n^n}{(n-d)^{n-d}}\widehat{\mathrm{vol}}(\imath(\overline{\eta}),Y_{\overline{\eta}},(\Delta_Y)_{\overline{\eta}})
\end{equation}
for every very general $x'\in Z$.

\smallskip

Suppose for the contradiction that 
\begin{equation} \label{mizmom}
\widehat{\mathrm{vol}}(\imath(x''),Y,\Delta_Y)>\frac{n^n}{(n-d)^{n-d}}\widehat{\mathrm{vol}}(\imath(\overline{\eta}),Y_{\overline{\eta}},(\Delta_Y)_{\overline{\eta}})
\end{equation}
for some $x''\in U$. Then, by applying the lower semicontinuity of normalized local volume to the projection $Y\times Z\to Z$ with the section $x'\mapsto (\imath(x'),x')$ (cf. \cite[Theorem 1]{BL21}), we obtain that there exists an open subscheme $U'\subseteq Z$ such that (\ref{mizmom}) holds for every closed point $x''\in U'$. It is a contradiction to (\ref{namduiduiaegu}). Hence, (\ref{namduiduiaegu}) holds for every $x''\in U$. Now, (\ref{Xu}) gives the proof.
\end{proof}

\begin{remark}
To ensure that $S_{\overline{\eta}}$ spreads out and that $\mathcal{O}_{Y,\imath(x')}/\mathfrak{b}_{m}$ is flat, it is necessary to shrink $Z$, which leads us to use the notion of ``very general''. For this reason, we must assume that the base field is uncountable.
\end{remark}

Let us prove the main theorem, Theorem \ref{maintheorem}.

\begin{proof}[Proof of Theorem \ref{maintheorem}]
Let $Z:=\overline{\{x\}}$. Since the theorem is local, we may assume that $X:=\Spec A$ is affine. After using generic smoothness (cf. \cite[Lemma 056V]{Stacks}), we may assume $Z$ is a smooth variety. Let $I$ be the ideal of $A$ corresponding to $Z$, and let us consider the henselization $(A^h,I^h)$ of $(A,I)$. By \cite[Lemma
0H74]{Stacks}, the identity $A/I\to A/I$ can be lifted to $A/I\to A^h$, and it is a section of $A^h\to A^h/I^h=A/I$. Since \cite[Lemma 00QO]{Stacks} holds, after shrinking $X$ and $Z$, there exists an étale neighborhood $\pi:Y\to X$ of $\eta_Z$ such that $Z\subseteq X$ lifts $\imath:Z\hookrightarrow Y$, and it has a section $Y\to Z$. Let $\Delta_Y:=\pi^*\Delta$. Note that
\begin{equation}\label{unjang}
\text{the natural map }\mathcal{O}_{Y,\imath(x)}\to \mathcal{O}_{Y_{x},\imath(x)}\text{ is an isomorphism}
\end{equation}
and
\begin{equation} \label{yanghak}
\text{The map }\mathcal{O}_{X,x}/\mathfrak{m}^m_x\to \mathcal{O}_{Y,\imath(x)}/\mathfrak{m}^m_{\imath(x)}\text{ is an isomorphism for every }m\ge 1,
\end{equation}
and the latter comes from \cite[Lemma 0AGU]{Stacks}.

\smallskip

Let us construct a $\Q$-Gorenstein family of klt singularities $Z\subseteq (Y',\Delta_{Y'})\to Z$.
\begin{itemize}
    \item After shrinking $Z$, we may assume $Y$ is flat over $Z$ by generic flatness (cf. \cite[Proposition 052A]{Stacks}).
    \item Since $Y\to X$ is étale, $Y$ is normal (cf. \cite[Lemma 034F]{Stacks}), and thus $\mathcal{O}_{Y_x,\imath(x)}\cong \mathcal{O}_{Y,\imath(x)}$ is normal (cf. (\ref{unjang})). Since an integral open subscheme $Y'$ of the connected component of $Y_x$ containing $\imath(x)$ is normal, by \cite[Lemma 0C3M]{Stacks}, the $Y'_{\overline{k(x)}}$ is also normal.
    \item By \cite[Proposition 11.3.1]{EGAIVIII}, after shrinking $Z$, we may assume that the connected component of $Y'_{\overline{z}}$ containing $\imath(\overline{z})$ is normal. After replacing $Z$ with the image of $Y'\setminus \Supp \Delta_{Y'}$ ($\Delta_{Y'}:=\Delta_Y|_{Y'}$), $\Supp \Delta_{Y'}$ does not contain $Y'_z$ for every $z\in Z$. So, $Z\subseteq (Y',\Delta_{Y'})\to Z$ is a $\Q$-Gorenstein family of singularities (cf. Definition \ref{nojeokbong}).
    \item Moreover, by Lemma \ref{idong}, there exists an open subscheme $Y''\subseteq Y'$ such that $Y''_{\overline{z}}$ is connected and contains $\imath(\overline{z})$. We replace $Y'$ with $Y''$. Then, $Z\subseteq (Y',\Delta_{Y'})\to Z$ is a $\Q$-Gorenstein family of irreducible singularities.
    \item Moreover, by \cite[Definition-Lemma 2.8]{Xu20}, after shrinking $Z$, we may assume that for every point $z\in Z$, $(Y'_{\overline{z}},(\Delta_{Y'})_{\overline{z}})$ is klt. Hence, $Z\subseteq (U_{Y'},\Delta_{Y'})\to Z$ is a $\Q$-Gorenstein family of klt singularities.
\end{itemize}

\smallskip

Note that by Lemma \ref{weolseong},
$$ \widehat{\mathrm{vol}}(x,X,\Delta)=\inf_{\nu\in \mathrm{DivVal}_{X,x}}A_{X,\Delta}(\nu)^{\dim X-\dim \overline{\{x\}}}\cdot \mathrm{vol}(\nu),$$
and applying \cite[2.14]{Kol13} ($A_{X,\Delta}(\nu)=A_{Y',\Delta_{Y'}}(\nu)$) and \cite[Proposition 2.24 (2)]{HLQ23} ($\mathrm{vol}(\nu)=\mathrm{vol}(\phi(\nu))$ for the homeomorphism $\phi:\mathrm{Val}_{X,x}\to \mathrm{Val}_{Y',\imath(x)}$ constructed in \cite[Proposition 2.24]{HLQ23}) gives
\begin{equation} \label{eungam} 
\widehat{\mathrm{vol}}(x,X,\Delta)=\widehat{\mathrm{vol}}(\imath(x),Y',\Delta_{Y'})\,\text{ and }\,\widehat{\mathrm{vol}}(x',X,\Delta)=\widehat{\mathrm{vol}}(\imath(x'),Y',\Delta_{Y'}).
\end{equation}
Note that \cite[Proposition 2.24 (2)]{HLQ23} holds even if $x\in X$ is not closed because $\widehat{\mathcal{O}_{X,x}}\cong \widehat{\mathcal{O}_{Y',\imath(x)}}$ (that follows from (\ref{yanghak})) and we can replace \cite[Lemma 2.14]{HLQ23} in \cite[Proof of Proposition 2.24 (1)]{HLQ23} with Lemma \ref{guryongpo}.

\smallskip

Let us prove
\begin{equation} \label{umuljae}
\widehat{\mathrm{vol}}(\imath(x),Y',\Delta_{Y'})=\widehat{\mathrm{vol}}(\imath(\overline{x}),Y'_{\overline{x}},(\Delta_{Y'})_{\overline{x}}).
\end{equation}
By \cite[End of Chapter 3]{BL21} (note that the quasi-monomiality of uniqueness of the computing valuation is proved in \cite{Xu20,XZ21}),
\begin{equation} \label{johang}
\widehat{\mathrm{vol}}(\imath(x),Y'_x,(\Delta_{Y'})_x)=\widehat{\mathrm{vol}}(\imath(\overline{x}),Y'_{\overline{x}},(\Delta_{Y'})_{\overline{x}}).
\end{equation}
The combination of (\ref{unjang}) and (\ref{johang}) gives (\ref{umuljae}).

\smallskip

By combining Lemma \ref{sambong} (applied to $Z\subseteq (Y',\Delta_{Y'})\to Z$) and (\ref{umuljae}), we obtain
$$ \widehat{\mathrm{vol}}(\imath(x),Y',\Delta_{Y'})=\frac{(n-d)^{n-d}}{n^n}\widehat{\mathrm{vol}}(\imath(x'),Y',\Delta_{Y'}),$$
for a general $x'\in Z$, and combining that with (\ref{eungam}), we prove the theorem.
\end{proof}

\bibliographystyle{habbvr}
\bibliography{biblio}

\end{document}